\documentclass[journal]{IEEEtran}

\usepackage{amssymb,verbatim,amsthm,soul} 
\usepackage{graphicx,lscape,tikz}
\usetikzlibrary{shapes,arrows}

\usepackage{caption,subcaption}
\usepackage{comment,color}
\specialcomment{notes}{\begingroup\sffamily\color{red}\noindent}{\endgroup}

\newtheorem{theorem}{Theorem} 

\newtheorem{assumption}{Assumption}
\newtheorem{definition}{Definition}

\usepackage{cite}

\ifCLASSINFOpdf
\else
\fi
\usepackage{amsmath}
\begin{document}
	
	\tikzstyle{dec} = [rectangle, draw, text width=7em, 
     text centered,  minimum height=4.5em]
	\tikzstyle{rea} = [rectangle, draw, 
	text width=5em, text centered, rounded corners, minimum height=4em]
	\tikzstyle{dotdot} = [rectangle,  
	text width=1em, text centered,  minimum height=1em]
	\tikzstyle{line} = [draw, -latex']
%
\title{The Value of Multi-stage Stochastic Programming in Risk-averse Unit Commitment under Uncertainty}
%
%
 
\author{Ali~\.{I}rfan~Mahmuto\u{g}ullar{\i},
	Shabbir~Ahmed,
	{\"O}zlem~{\c{C}}avu{\c{s}}
	and~M.~Selim~Akt{\"u}rk
	\thanks{A.~\.{I}.~Mahmuto\u{g}ullar{\i}, {\"O}.~{\c{C}}avu{\c{s}} and M.~S.~Akt{\"u}rk are with Department
		of Industrial Engineering, Bilkent University, Ankara, 06800, Turkey. e-mail: a.mahmutogullari@bilkent.edu.tr, ozlem.cavus@bilkent.edu.tr, akturk@bilkent.edu.tr}
	 \thanks{S.~Ahmed is with H. Milton Stewart School of Industrial and Systems Engineering, Georgia Institute of Technology, Atlanta, 30318, GA, USA. e-mail: sahmed@isye.gatech.edu}
	 \thanks{The first author is supported by the Scientific and Technological Research Council of Turkey (T\"{U}B\.{I}TAK) program number B\.{I}DEB-2214-A. The second author is supported by the National Science Foundation Grant 1633196.}
	\thanks{Manuscript submitted \today.}}

\maketitle

\begin{abstract}
Day-ahead scheduling of electricity generation or unit commitment is an important and challenging optimization problem in power systems. Variability in net load arising from the increasing penetration of renewable technologies have motivated study of various classes of stochastic 
unit commitment models. In two-stage models, the generation schedule for the entire day is fixed while the dispatch is adapted to the uncertainty, whereas in multi-stage models the generation schedule is also allowed to dynamically adapt to the uncertainty realization. Multi-stage models provide more flexibility in the generation schedule, however, they  require significantly higher computational effort than two-stage models. To justify this additional computational effort, we provide theoretical and empirical analyses of the \emph{value of multi-stage solution}  
for risk-averse multi-stage stochastic unit commitment models. The value of multi-stage solution measures the relative advantage of multi-stage solutions over their two-stage counterparts.
Our results indicate that, for unit commitment models, value of multi-stage solution increases with the level of uncertainty and number of periods, and decreases with the degree of risk aversion of the decision maker.

\end{abstract}



%
\IEEEpeerreviewmaketitle

\section{Introduction} \label{sec:intro}
\IEEEPARstart{U}{}\emph{nit commitment} (UC) is a challenging optimization problem used for day-ahead generation scheduling given net load forecasts and various operational constraints~\cite{kazarlis1996genetic}. The output schedule includes on-off status of generators and the production amounts, called \emph{economic dispatch}~\cite{huang2017electrical}, for every time step.




There has been a great deal of research on  deterministic UC models where the problem parameters are assumed to be known exactly\cite{padhy2004unit}.  These models cannot capture variability and uncertainty. Common sources of uncertainty are departures from forecasts and unreliable equipment. The departures from forecasts generally stem from the variability in net load and production amounts, whereas unreliable equipment may result in generator and transmission line outages \cite{huang2017electrical}, \cite{ruiz2009uncertainty}. 
The penetration of renewable energy has increased the volatility of power systems in recent years. The production amount of energy from wind and solar power are not controllable but can only be forecasted \cite{brown2008optimization}.
 
\emph{Robust optimization} and \emph{stochastic programming} are two common frameworks used to address the uncertainty in UC problems. In robust optimization models, it is assumed that the uncertain parameters take values in some uncertainty sets and the objective is to minimize the worst case cost (cf. \cite{bertsimas2013adaptive}, \cite{lorca2017multistage}, \cite{zhao2012robust}, \cite{jiang2012robust} and \cite{wang2013two}). In stochastic programming models, the uncertainty is represented by a probability distribution (cf. \cite{cheung2015toward}, \cite{tahanan2015}, \cite{papavasiliou2015applying}, \cite{wang2008security} and \cite{takriti1996stochastic}). 
 In \textit{two-stage} stochastic programming UC models, the generation schedule is fixed for the entire day before the beginning of the day while dispatch is adapted to uncertainty as in \cite{caroe1998two,wang2012chance} and  \cite{zheng2013decomposition}. On the other hand, in \textit{multi-stage} stochastic programming UC models both the generation schedule and dispatch are allowed to dynamically adapt to uncertainty realization at each hour (see for example, \cite{takriti1996stochastic,takayuki2004stochastic} and \cite{jiang2016cutting}). Therefore, they incorporate multistage forecasting information with varying accuracy and express relation between time periods appropriately. However, in general, the multi-stage models are computationally difficult. A detailed comparison of two- and multi-stage models can be found in \cite{zheng2015stochastic} and \cite{lorca2016multistage}.

The computational challenge of multi-stage models motivates the question on whether the effort to solve them is worthwhile. In \cite{huang2009value}, this question is addressed for a risk-neutral stochastic capacity planning problem. In the present paper, we address this question for risk-averse UC (RA-UC) problems where the objective is a dynamic measure of risk. We provide theoretical and empirical analysis on the value of the \textit{multi stage solution} (VMS) where VMS measures the relative advantage to solve the multi-stage models over their two-stage counterparts. 
 
The rest of the paper is organized as follows: In Section \ref{sec:problem}, we define the RA-UC problem and present two- and multi-stage stochastic models. In Section \ref{sec:vms}, we define VMS and provide analytical bounds for it. In Section \ref{sec:computation}, we present results of  computational experiments. In Section \ref{sec:conclusion}, we discuss possible future extensions of the current work.

\section{Risk-averse Unit Commitment Problem} \label{sec:problem}
\subsection{Deterministic UC formulation}
We first present an abstract deterministic formulation of the UC problem. Let $I$ be the number of generators and $T$ be the number of periods. Also, let $\mathcal{I} := \{1,\ldots,I\}$ and $\mathcal{T} := \{1,\ldots,T\}$ be the sets of generators and time periods, respectively. A formulation of the UC problem is as follows:
\begin{align}
\text{min} \; & \sum_{t=1}^T f_t(\boldsymbol{u}_t, \boldsymbol{v}_t, \boldsymbol{w}_t) \label{det-obj}\\
\text{s.t.} \; &  \sum_{i = 1}^I v_{it} \geq d_t, \; \forall t \in \mathcal{T}  \label{det-dem} \\ 
& \underline{q}_i u_{it} \leq v_{it} \leq \overline{q}_i u_{it}, \; \forall i \in \mathcal{I}, t \in \mathcal{T}  \label{det-cap} \\
& (\boldsymbol{u}_{1},\boldsymbol{v}_{1},\boldsymbol{w}_{1}) \in \mathcal{X}_1, \label{det-set1} \\
& (\boldsymbol{u}_{t},\boldsymbol{v}_{t},\boldsymbol{w}_{t}) \in \mathcal{X}_t(\boldsymbol{u}_{t-1},\boldsymbol{v}_{t-1},\boldsymbol{w}_{t-1}), \;  \forall t \in \mathcal{T} \setminus \{1\} \label{det-set} \\
& \boldsymbol{u}_t \in \{0,1\}^I, \boldsymbol{v}_t \in \mathbb{R}_+^I, \boldsymbol{w}_t \in \mathbb{R}^k, \; \forall t \in \mathcal{T}  \label{det-dom} 
\end{align}
Decision variables $u_{it}$ and $v_{it}$ represent the binary on/off status and production of generator $i \in \mathcal{I}$ in period $t \in \mathcal{T}$, respectively. The bold symbols $\boldsymbol{u}_t := (u_{1t},u_{2t},\ldots,u_{It})$ and $\boldsymbol{v}_t := (v_{1t},v_{2t},\ldots,v_{It})$ are the vectors of status and production decisions in period $t \in \mathcal{T}$, respectively. The vector $\boldsymbol{w}_t$ denotes auxiliary variables associated with  period $t \in \mathcal{T}$. These variables are used model various operational constraints.
The objective (\ref{det-obj})  is the sum of production, start-up and shut-down costs in all periods. The function $f_t(\cdot)$  represents the total cost in a period $t \in \mathcal{T}$. Constraint (\ref{det-dem}) ensures satisfaction of the power demand. Constraint (\ref{det-cap}) enforces lower and upper production limits on the generators. Other operational restrictions are represented by constraints (\ref{det-set1}) and  (\ref{det-set}). The  temporal relationship between consecutive periods such as start-up, rump-up, shut-down and rump-down restrictions are modeled  by the set constraint (\ref{det-set}). Domain restrictions of the decision variables are given by constraint (\ref{det-dom}). A concrete version of the above abstract formulation is presented in Appendix \ref{app:model}.


\subsection{Uncertainty and Risk models}
In the deterministic formulation above, net load values are assumed to be known exactly. This is a restrictive assumption in practice. We assume that the net load is random and denoted by a random variable $\widetilde{d}_t$ in period $t \in \mathcal{T}$ from a probability space $(\Omega,\mathcal{F},P)$. Here $\Omega$ is a sample space equipped with sigma algebra $\mathcal{F}$ and probability measure $P$. An element of the sample space $\Omega$ is called as a \textit{scenario} (or a sample path) and represents a possible realization of the net load values in all periods. The sequence of sigma algebras $ \{\emptyset,\Omega\}=\mathcal{F}_1 \subseteq \mathcal{F}_2 \subseteq \cdots \subseteq \mathcal{F}_T = \mathcal{F}$ is  called as a \textit{filtration} and it represents the gradually increasing information through the decision horizon $1,2,\ldots,T$. The set of $\mathcal{F}_t-$measurable random variables is denoted by $\mathcal{Z}_t$ for $t \in \mathcal{T}$.  The random demand $\widetilde{d}_t$ in period $t$ is $\mathcal{F}_t-$measurable, that is $\widetilde{d}_t \in \mathcal{Z}_t$ for $t \in \mathcal{T}$. Note that since $\mathcal{F}_1 = \{\emptyset,\Omega\}$ by definition, $\mathcal{Z}_1 = \mathbb{R}$ and the demand in the first period is deterministic. 

To extend the deterministic UC model to this ucertainty setting, we have that  the  decisions in period $t$ to depend on realization of the history of net load process $\widetilde{d}_{[t]} := (\widetilde{d}_{1},\ldots,\widetilde{d}_{t})$ up to period $t$. Therefore, we use the $\mathcal{F}_t-$measurable vectors $\widetilde{\boldsymbol{u}}_t(\widetilde{d}_{[t]})$, $\widetilde{\boldsymbol{v}}_t(\widetilde{d}_{[t]})$ and $\widetilde{\boldsymbol{w}}_t(\widetilde{d}_{[t]})$ to represent status, production and auxiliary decisions in period $t \in \mathcal{T}$, respectively. 
The total cost at period $t$ is also $\mathcal{F}_t-$measurable, i.e., $f_t(\widetilde{\boldsymbol{u}}_t(\widetilde{d}_{[t]}), \widetilde{\boldsymbol{v}}_t(\widetilde{d}_{[t]}), \widetilde{\boldsymbol{w}}_t(\widetilde{d}_{[t]}))\in \mathcal{Z}_t$. We use conditional risk measures in order to quantify the risk involved in a random cost at period $t+1$ based on the available informations at period $t$ for $t \in  \mathcal{T}\setminus\{T\}$. The mapping $\rho_t : \mathcal{Z}_{t+1} \rightarrow \mathcal{Z}_t$ is called a \textit{conditional risk measure} if it satisfies the following four axioms of coherent risk measures (the subscript $t$ is suppressed for notational brevity): 
\begin{itemize}
	\item[(A1)] \emph{Convexity}: $\rho(\alpha Z + (1-\alpha)W) \leq \alpha \rho(Z) + (1-\alpha)\rho(W)$ for all $Z,W \in \mathcal{Z}$ and $\alpha \in [0,1]$,
	\item[(A2)] \emph{Monotonicity}: $Z \succeq W$ implies $\rho(Z) \geq \rho(W)$ for all $Z,W \in \mathcal{Z}$,
	\item[(A3)] \emph{Translational Equivariance}: $\rho(Z + c) = \rho(Z) + c$ for all $c \in \mathbb{R}$ and $Z \in \mathcal{Z}$,
	\item[(A4)] \emph{Positive Homogeneity}: $\rho(cZ) = c\rho(Z)$ for all $c > 0$ and $Z \in \mathcal{Z}$,
\end{itemize}
where $Z \succeq W$ indicates point-wise partial ordering defined on set  $\mathcal{Z}$.
See \cite{artzner1999coherent} and \cite{shapiro2009lectures} for a detailed discussions on coherent  and conditional risk measures. An example of a conditional risk measure is the \textit{conditional mean-upper semi deviation} 
\begin{equation}\label{musd}
  \rho_t(Z_{t+1}) = \mathbb{E}[Z_{t+1}|\mathcal{F}_t] + \lambda\mathbb{E}[(Z_{t+1}-\mathbb{E}[Z_{t+1}|\mathcal{F}_t])_+|\mathcal{F}_t], 
\end{equation}
where $\mathbb{E}[\cdot|\mathcal{F}_t]$ is the conditional expectation with respect to the sigma algebra $\mathcal{F}_t$, $\lambda \in [0,1]$ is a parameter controlling the degree of risk aversion and $(\cdot)_{+}$ is the positive part function for all $Z_{t+1} \in \mathcal{Z}_{t+1}$. 

The objective of the risk averse UC (RA-UC) problem is to minimize the risk involved with the cost sequence $\{Z_t\}_{t=1}^T$ where $Z_t := f_t(\widetilde{\boldsymbol{u}}_t(\widetilde{d}_{[t]}), \widetilde{\boldsymbol{v}}_t(\widetilde{d}_{[t]}), \widetilde{\boldsymbol{w}}_t(\widetilde{d}_{[t]}))$ is a shorthand notation for the total cost in period $t \in \mathcal{T}$. Thus, as in \cite{collado2012scenario,shapiro2009lectures}, we define the dynamic coherent risk measure $\varrho: \mathcal{Z}_1 \times \mathcal{Z}_2 \times \cdots \times \mathcal{Z}_T \rightarrow \mathbb{R}$ by using nested composition of the conditional risk measures $\rho_{1}(\cdot),\rho_{2}(\cdot),\ldots,\rho_{T-1}(\cdot)$, that is,
\begin{equation*} 
\varrho(Z_1,Z_2,\ldots,Z_T) :=   Z_1 + \rho_1(Z_2     + \cdots \rho_{T-1}(Z_T) \cdots )
\end{equation*}
is the risk associated with this cost sequence.  Due to translational equivariance property of conditional risk measures, we have an alternative representation of the dynamic coherent measure of risk $\varrho(\cdot)$ as
\begin{equation}\label{rhobar} 
 \rho\left(\sum_{t=1}^T Z_t \right) :=  
 \varrho(Z_1,Z_2,\ldots,Z_T) 
\end{equation}
where  $\rho = \rho_1 \circ \rho_2 \circ \cdots \circ \rho_{T-1} : \mathcal{Z} \rightarrow \mathbb{R}$ is called as a \emph{composite risk measure} and $\mathcal{Z} := \mathcal{Z}_T$. The composite risk measure  $\rho(\cdot)$ satisfies the coherence axioms (A1)-(A4).
 Therefore, $\rho(\cdot)$ is a coherent risk measure as shown in \cite[Eqn. 6.234]{shapiro2009lectures}.

\subsection{Two-stage and Multi-stage models}
We consider two different models for the RA-UC problem. In the \textit{two-stage model}, the on/off status decisions are fixed at the beginning of the day and production (or dispatch) decisions are adapted to uncertainty in the random demand. On the other hand, in the \textit{multi-stage model}, both the status and production decisions are fully adapted to uncertainty in net load. In order to clarify the distinction between two models, the decision dynamics in the two- and multi-stage models are depicted as in \figurename{ \ref{fig:decpro2}} and \figurename{ \ref{fig:decprom}}, respectively. 
\begin{figure*}[h!]
	\centering				  
	\caption{Order of decisions in the two-stage model. \label{fig:decpro2}}
	\resizebox{0.6\textwidth}{!}{
		\begin{tikzpicture}[node distance = 2cm, auto]
		\node[dec](d1){Decide \\ $\{\boldsymbol{u}_t\}_{t=1}^T,\boldsymbol{v}_1,\boldsymbol{w}_1$};
		\node [rea, right of=d1 , node distance=3cm] (r1) {Observe \\ $\widetilde{d}_2$};
		\node [dec, right of=r1 , node distance=3cm] (d2) {Decide \\ $\boldsymbol{v}_2,\boldsymbol{w}_2$};
		\node [dotdot,right of=d2, node distance=2cm] (dot){...};
		\node [rea, right of=dot , node distance=2cm] (rT) {Observe \\ $\widetilde{d}_T$};
		\node [dec, right of=rT , node distance=3cm] (dT) {Decide \\ $\boldsymbol{v}_T,\boldsymbol{w}_T$};
		\path [line] (d1) -- (r1);
		\path [line] (r1) -- (d2);
		\path [line] (d2) -- (dot);
		\path [line] (dot) -- (rT);
		\path [line] (rT) -- (dT);
		\end{tikzpicture}
	}
\end{figure*}
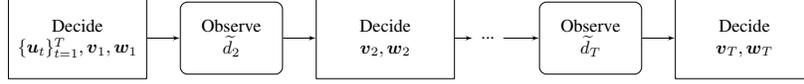
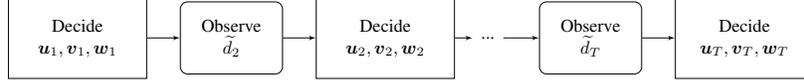
\begin{figure*}[h!]
	\centering
	\caption{Order of decisions in the multi-stage model. \label{fig:decprom}}
	\resizebox{0.6\textwidth}{!}{
		\begin{tikzpicture}[node distance = 2cm, auto]
		\node[dec](d1){Decide \\ $\boldsymbol{u}_1,\boldsymbol{v}_1,\boldsymbol{w}_1$};
		\node [rea, right of=d1 , node distance=3cm] (r1) {Observe \\ $\widetilde{d}_2$};
		\node [dec, right of=r1 , node distance=3cm] (d2) {Decide \\ $\boldsymbol{u}_2,\boldsymbol{v}_2,\boldsymbol{w}_2$};
		\node [dotdot,right of=d2, node distance=2cm] (dot){...};
		\node [rea, right of=dot , node distance=2cm] (rT) {Observe \\ $\widetilde{d}_T$};
		\node [dec, right of=rT , node distance=3cm] (dT) {Decide \\ $\boldsymbol{u}_T,\boldsymbol{v}_T,\boldsymbol{w}_T$};
		\path [line] (d1) -- (r1);
		\path [line] (r1) -- (d2);
		\path [line] (d2) -- (dot);
		\path [line] (dot) -- (rT);
		\path [line] (rT) -- (dT);
		\end{tikzpicture}
	}
\end{figure*}

The two-stage model (TS) for the RA-UC problem is given as
\begin{align}
\text{min} \; &\rho\left[\sum_{t=1}^T f_t(\boldsymbol{u}_t, \widetilde{\boldsymbol{v}}_t(\widetilde{d}_{[t]}), \widetilde{\boldsymbol{w}}_t(\widetilde{d}_{[t]}))\right] \label{two-obj}   \\
\text{s.t.} \; &  \sum_{i \in \mathcal{I}} \widetilde{v}_{it}(\widetilde{d}_{[t]}) \geq \widetilde{d}_t, \; \forall t \in \mathcal{T}  \label{two-dem}  \\ 
& \underline{q}_i u_{it} \leq \widetilde{v}_{it}(\widetilde{d}_{[t]}) \leq \overline{q}_i u_{it}, \; \forall i \in \mathcal{I}, t \in \mathcal{T}  \label{two-cap} \\
&(\boldsymbol{u}_{1},\boldsymbol{v}_{1},\boldsymbol{w}_{1}) \in \mathcal{X}_1 \label{two-set1} \\
& (\boldsymbol{u}_{t},\widetilde{\boldsymbol{v}}_{t}(\widetilde{d}_{[t]}),\widetilde{\boldsymbol{w}}_{t}(\widetilde{d}_{[t]})) \in   \nonumber\\
&  \mathcal{X}_t( \boldsymbol{u}_{t-1},\widetilde{\boldsymbol{v}}_{t-1}(\widetilde{d}_{[t-1]}),\widetilde{\boldsymbol{w}}_{t-1}(\widetilde{d}_{[t-1]}), \widetilde{d}_{[t]}), \; \forall t \in \mathcal{T} \setminus \{1\} \label{two-set} \\
& \boldsymbol{u}_t \in \{0,1\}^I, \widetilde{\boldsymbol{v}}_t(\widetilde{d}_{[t]}) \in \mathbb{R}_+^I, \widetilde{\boldsymbol{w}}_t(\widetilde{d}_{[t]}) \in \mathbb{R}^k, \; \forall t \in \mathcal{T}  \label{two-dom} 
\end{align}
The objective (\ref{two-obj}) of TS is the composite risk measure defined in (\ref{rhobar}) applied to the total cost sequence. The  inequalities  (\ref{two-dem}) and (\ref{two-cap}) are analogous to the constraints (\ref{det-dem}) and (\ref{det-cap}), respectively. The set constraint (\ref{two-set1}) is identical to (\ref{det-set1}) since the net load in the first period is deterministic. In constraint (\ref{two-set}), $\mathcal{X}_t$ is an $\mathcal{F}_t-$measurable feasibility set. The domain constraint (\ref{two-dom}) states that only production and auxiliary decisions depend on the demand history and the status decisions are deterministic. However, in the multi-stage model of the RA-UC problem, all decisions are made based on the history. Hence, the multi-stage model (MS) can be written as  
\begin{align}
\text{min} \; &\rho\left[\sum_{t=1}^T f_t(\widetilde{\boldsymbol{u}}_t(\widetilde{d}_{[t]}), \widetilde{\boldsymbol{v}}_t(\widetilde{d}_{[t]}), \widetilde{\boldsymbol{w}}_t(\widetilde{d}_{[t]}))\right] \label{mul-obj}   \\ 
\text{s.t.} \; &  \sum_{i \in \mathcal{I}} \widetilde{v}_{it}(\widetilde{d}_{[t]}) \geq \widetilde{d}_t, \; \forall t \in \mathcal{T}  \label{mul-dem}  \\ 
\text{s.t.} \;  & \underline{q}_i \widetilde{u}_{it}(\widetilde{d}_{[t]}) \leq \widetilde{v}_{it}(\widetilde{d}_{[t]}) \leq \overline{q}_i \widetilde{u}_{it}(\widetilde{d}_{[t]}), \; \forall i \in \mathcal{I}, t \in \mathcal{T}  \label{mul-cap} \\
&(\boldsymbol{u}_{1},\boldsymbol{v}_{1},\boldsymbol{w}_{1}) \in \mathcal{X}_1 \label{mul-set1} \\ 
& (\widetilde{\boldsymbol{u}}_{t}(\widetilde{d}_{[t]}),\widetilde{\boldsymbol{v}}_{t}(\widetilde{d}_{[t]}),\widetilde{\boldsymbol{w}}_{t}(\widetilde{d}_{[t]})) \in   \nonumber\\
&  \mathcal{X}_t( \widetilde{\boldsymbol{u}}_{t-1}(\widetilde{d}_{[t-1]}),\widetilde{\boldsymbol{v}}_{t-1}(\widetilde{d}_{[t-1]}),\widetilde{\boldsymbol{w}}_{t-1}(\widetilde{d}_{[t-1]}), \widetilde{d}_{[t]}),  \nonumber \\
& \forall t \in \mathcal{T} \setminus \{1\} \label{mul-set} \\
& \widetilde{\boldsymbol{u}}_t(\widetilde{d}_{[t]})  \in \{0,1\}^I, \widetilde{\boldsymbol{v}}_t(\widetilde{d}_{[t]}) \in \mathbb{R}_+^I, \widetilde{\boldsymbol{w}}_t(\widetilde{d}_{[t]}) \in \mathbb{R}^k,\nonumber \\
&  \forall t \in \mathcal{T}  \label{mul-dom}  
\end{align}
Note that the multi-stage model MS is identical with TS except the status decisions are fully adaptive to the random net load process. 

An optimal solution of either TS and MS is a policy that minimizes the value of the dynamic coherent risk measure. Both in TS and MS, the optimality of a policy should only be with respect to possible future realizations given the available information at the time when the decision is made. This principle is called as \emph{time consistency}. In \cite[Example 2]{shapiro2009time}, it is shown that time consistency enables us to use the composite risk measure in minimization among all possible decisions instead of nested minimizations in a dynamic coherent measure of risk. 

\section{Value of The Multi-stage Solution} \label{sec:vms}
Although an optimal solution of MS provides more flexible day-ahead schedule with respect to different realizations of parameters, the number of binary variables in MS is proportional to $\mathcal{N} \times I$ where $\mathcal{N}$ is the number of possible demand realizations in all periods if $\Omega$ is finite. However, the number of binary variables in TS is proportional to $T \times I$. Since $\mathcal{N} >> T$ for any non-trivial problem, computational difficulty of MS is significantly more than TS. Therefore, it is important to figure out if the additional effort to solve MS is worthwhile. We define the VMS in order to quantify the relative advantage of the multi-stage solution over their two-stage counterparts.
\begin{definition}
	The value of multi-stage solution (VMS) is the difference between the optimal values of TS and MS, that is, $\text{VMS} = z^{TS}-z^{MS}$ where $z^{TS}$ and $z^{MS}$ are the optimal values of TS and MS, respectively.
\end{definition} 
Since an optimal solution of MS provides more flexibility in status decisions with respect to uncertain net load realizations, we have $z^{TS} \geq z^{MS}$ and therefore  $\text{VMS} \geq 0$. Next we provide theoretical bounds on the VMS under some assumptions. 

\begin{assumption}\label{as:recourse}
	There exists a generator $j^* \in \mathcal{I}$  such that $\underline{q}_{j^*}\leq \widetilde{d}_t \leq \overline{q}_{j^*}$ with probability 1 and with no minimum start up and shut down time for each  $t \in \mathcal{T}$. 
\end{assumption}

\begin{assumption}\label{as:boundeddemand}
	There exists an upper bound  $d_t^{\max} \in \mathbb{R}_+$ on the net load values such that  $0 \leq \widetilde{d}_t \leq d_t^{\max}$ with  probability 1 for each $t \in \mathcal{T}$.
\end{assumption}

\begin{assumption}\label{as:linearcost}
The production cost of the each generator $i \in \mathcal{I}$ is linear and stationary, and there are no start-up and shut-down costs. In this case the total cost function in each period is of the form
$ f_t(\boldsymbol{u}_t, \boldsymbol{v}_t, \boldsymbol{w}_t) =  \sum_{i \in \mathcal{I}} (a_{i} u_{it} + b_i v_{it})$
for some positive coefficients $a_i$ and $b_i$ for all $i \in \mathcal{I}$.
\end{assumption}

Assumption \ref{as:recourse} ensures that TS and MS always have at least one feasible solutions and therefore both problems have \textit{complete recourse}.  Assumption \ref{as:boundeddemand} states that the net load in each period is bounded.  We also define $\widetilde{D} := \sum_{t =1}^T \widetilde{d}_t$ as the total net load  and $D^{\max} := \sum_{t =1}^T d^{\max}_t$ as an upper bound on $\widetilde{D}$. The above assumptions are somewhat restrictive but necessary for the analytical result next. In Section \ref{sec:computation}, we will provide numerical results 
showing that the analytical results hold even without these assumptions.

\begin{theorem}\label{thepro}
Under Assumptions 1, 2 and 3 we have that 
	\begin{equation*}
		\alpha_*D^{\max} - \alpha^* \rho( \widetilde{D}) \leq \text{ \emph{VMS} } \leq \alpha^* D^{\max} - \alpha_* \rho( \widetilde{D}).
	\end{equation*}
where
	\begin{align}
	\alpha_{*} &:= \underset{i \in \mathcal{I}}{\min} \left\lbrace a_{i} +b_{i}\underline{q}_i\right\rbrace \Big /     \underset{i \in \mathcal{I}}{\max} \left\lbrace\overline{q}_i\right\rbrace \text{ and} \nonumber \\ 
	\alpha^{*} &:= \underset{i \in \mathcal{I}}{\max} \left\lbrace a_{i} +b_{i}\overline{q}_i\right\rbrace \Big /     \underset{i \in \mathcal{I}}{\min} \left\lbrace\underline{q}_i\right\rbrace \nonumber \nonumber 
	\end{align}
	are cost related problem parameters. 
\end{theorem}

\begin{proof}
	Assumption \ref{as:recourse} implies that both TS and MS are feasible. Since the net loads are bounded due to Assumption \ref{as:boundeddemand}, both models have at least one optimal solution.
	
	Let  $\{\widetilde{\boldsymbol{u}}^*_t,\widetilde{\boldsymbol{v}}^*_t,\widetilde{\boldsymbol{w}}^*_t\}_{t \in \mathcal{T}}$ be an optimal policy obtained by  solving the multi-stage model MS. By Assumption~\ref{as:linearcost}, we have $f_t(\widetilde{\boldsymbol{u}}^*_t(\widetilde{d}_{[t]}),\widetilde{\boldsymbol{v}}^*_t(\widetilde{d}_{[t]}),\widetilde{\boldsymbol{w}}^*_t(\widetilde{d}_{[t]})) = \sum_{t \in \mathcal{T}} \sum_{i \in \mathcal{I}} a_{i}\widetilde{u}^*_{it}(\widetilde{d}_{[t]}) + b_{i}\widetilde{v}^*_{it}(\widetilde{d}_{[t]})$
	
	For a realization $d_1,d_2,\ldots,d_T$ of the random net load process $\widetilde{d}_1,\widetilde{d}_2,\ldots,\widetilde{d}_T$, let $[\boldsymbol{u}^*_t,\boldsymbol{v}^*_t] := [\widetilde{\boldsymbol{u}}^*_t,\widetilde{\boldsymbol{v}}^*_t](d_{[t]})$ be the optimal status and production decisions  for $t \in \mathcal{T}$. Then,  we have 
	\begin{align}
    & \sum_{t \in \mathcal{T}} \sum_{i \in \mathcal{I}}  a_{i}u^*_{it} + b_{i}v^*_{it}  \geq \sum_{t \in \mathcal{T}} \sum_{i \in \mathcal{I}}  a_{i}u^*_{it} + b_{i}\underline{q}_i u^*_{it}  \nonumber  \\     
	& = \sum_{t \in \mathcal{T}} \sum_{i \in \mathcal{I}}  (a_{i} + b_{i}\underline{q}_i)u^*_{it} \geq  \underset{i \in \mathcal{I}}{\min} \{a_{i} + b_{i}\underline{q}_i\} \sum_{t \in \mathcal{T}} \sum_{i \in \mathcal{I}} u^*_{it} \nonumber \\
	& \geq  \underset{i \in \mathcal{I}}{\min} \{a_{i} + b_{i}\underline{q}_i\} \sum_{t \in \mathcal{T}} \sum_{i \in \mathcal{I}} \frac{v^*_{it}}{\overline{q}_i}  \geq  \frac{	\underset{i \in \mathcal{I}}{\min} \{a_{i} + b_{i}\underline{q}_i\}}{\underset{i \in \mathcal{I}}{\max} \{\overline{q}_i\}} \sum_{t \in \mathcal{T}} \sum_{i \in \mathcal{I}}  v^*_{it} \nonumber \\ 
	& = a_*  \sum_{t \in \mathcal{T}} \sum_{i \in \mathcal{I}}  v^*_{it}  \geq a_*  \sum_{t \in \mathcal{T}} d_t \nonumber 
	\end{align}
where the first, third and fifth inequalities follow from feasibility. Since $\sum_{t \in \mathcal{T}} \sum_{i \in \mathcal{I}}  a_{i}u^*_{it} + b_{i}v^*_{it} \geq a_*  \sum_{t \in \mathcal{T}} d_t$ for any sample path  $d_1,d_2,\ldots,d_T$, we have $\sum_{t \in \mathcal{T}} \sum_{i \in \mathcal{I}}  a_{i}\widetilde{\boldsymbol{u}}^*_{it}(d_{[t]}) + b_{i}\widetilde{\boldsymbol{v}}^*_{it}(d_{[t]}) \succeq a_*  \sum_{t \in \mathcal{T}} \widetilde{d}_t = \alpha_{*}\widetilde{D}$.
Due to the monotonicity axiom (A2) and positive homogeneity axiom (A4), we get
\begin{align}
z^{MS}  & = \rho\left(\sum_{t \in \mathcal{T}} \sum_{i \in \mathcal{I}}  a_{i}\widetilde{u}^*_{ti}(d_{[t]}) + b_{i}\widetilde{v}^*_{ti}(d_{[t]}) \right) \nonumber \\ 
& \geq \rho(\alpha_{*}\widetilde{D}) = \alpha_{*} \rho(\widetilde{D}). \nonumber
\end{align}
Next, we consider a feasible policy $\{\widehat{\boldsymbol{u}}^*_t,\widehat{\boldsymbol{v}}^*_t,\widehat{\boldsymbol{w}}^*_t\}_{t \in \mathcal{T}}$ to the multi-stage model where $\widehat{u}_{j^*t}(\widetilde{d}_{[t]})=1$, $\widehat{v}_{j^*t}(\widetilde{d}_{[t]})=\widetilde{d}_t$ and all other status and generation variables are set to zero for a sample path $d_1,d_2,\ldots,d_t$. The feasibility of the solution is guaranteed by Assumption \ref{as:recourse}. Then, 
\begin{align}
& z^{MS}  \leq  \rho \left(  \sum_{t \in\mathcal{T}} \sum_{i \in\mathcal{I}} a_{i}\widehat{u}_{it}(\widetilde{d}_{[t]}) + b_{i}\widehat{v}_{it}(\widetilde{d}_{[t]}) \right) \nonumber \\ 
& = \rho \left( \sum_{t \in\mathcal{T}} a_{j^*}\widehat{u}_{j^*t}(\widetilde{d}_{[t]}) + b_{j^*}\widehat{v}_{j^*t}(\widetilde{d}_{[t]})\right)  = \rho \left( \sum_{t \in\mathcal{T}} a_{j^*} + b_{j^*} \widetilde{d}_{t}\right) \nonumber \\
& = \rho \left( \sum_{t \in\mathcal{T}} \frac{a_{j^*} + b_{j^*} \widetilde{d}_{t}}{\widetilde{d}_{t}}\widetilde{d}_{t} \right) \leq \rho \left( \sum_{t \in\mathcal{T}} \frac{a_{j^*} + b_{j^*} \overline{q}_{j^*}}{\underline{q}_{j^*}}\widetilde{d}_{t} \right) \nonumber \\
& \leq \frac{ \underset{i \in \mathcal{I}}{\max} \{a_{i} + b_{i} \overline{q}_i\}}{\underset{i \in \mathcal{I}}{\min}\{\underline{q}_i\}} \rho \left( \sum_{t \in\mathcal{T}} \widetilde{d}_{t} \right)  = \alpha^* \rho \left( \sum_{t \in\mathcal{T}} \widetilde{d}_{t} \right)  \leq \alpha^* \rho (\widetilde{D}) \nonumber
\end{align}
where the first inequality follows from feasibility, the second inequality follows from Assumption \ref{as:recourse} and the third equality follows from axiom (A4) and the definition of $\alpha^{*}$. Thus, we get lower and upper bounds for the multi-stage problems, that is, 
\begin{equation}\label{multibounds}
\alpha_* \rho(\widetilde{D}) \leq z^{MS} \leq  \alpha^* \rho(\widetilde{D}). 
\end{equation}
Note that in the two-stage model, the status decisions in period $t \in \mathcal{T}$ is identical for all realizations of problem parameters in that period and satisfies $\max\{\widetilde{v}^*_{it}(\widetilde{d}_{[t]})\} \leq \overline{q}u^*_{it}$. Then, using this fact, a similar analysis can be used to obtain lower and upper bounds for the two-stage model and we get
\begin{equation}\label{twobounds}
\alpha_* D^{\max} \leq z^{TS} \leq  \alpha^* D^{\max}. 
\end{equation}
The claim of the theorem  follows from (\ref{multibounds}) and (\ref{twobounds}).  
\end{proof} 

If the generators are almost identical and lower and upper production limits are close enough, we have $\alpha_{*} \approx \alpha \approx \alpha^*$. Then, we have
\begin{equation}\label{almost}
	\text{VMS} \approx \alpha(D^{\max} - \rho(\widetilde{D})).
\end{equation}
Note that $ 0 \leq \rho(\widetilde{D}) \leq D^{\max}$ and the approximation (\ref{almost}) implies that the VMS increases with $D^{\max}$ and therefore variability in the net load. However, for fixed variability, the VMS decreases with $\rho(\widetilde{D})$ and therefore the degree of risk aversion.

Assume that the net load in period $t \in \mathcal{T}$ is $\widetilde{d}_t = \overline{d}_{t} + \mathcal{U}[-\Delta,\Delta]$ where $\overline{d}_{t}$ is a deterministic value and $\mathcal{U}[-\Delta,\Delta]$ is an error term uniformly distributed between $-\Delta$ and $\Delta$ for some $\Delta \in \mathbb{R}_+$. Also assume that the composite risk measure $\rho(\cdot)$ is obtained using conditional mean-upper semi deviation as given in (\ref{musd}). Then, 

\begin{align}
    \text{VMS} & \approx  \alpha(D^{\max} - \rho(\widetilde{D})) \nonumber \\ 
     & =  \alpha \left( \sum_{t = 1}^T d^{\max}_t - \rho\left(\sum_{t = 1}^T \widetilde{d}_t\right)\right) \nonumber \\
     & =   \alpha T \left(1-\frac{\lambda}{4}\right) \Delta \label{final}
\end{align}
where the second equality follows from definitions of $d^{\max}_t$, $\widetilde{d}_t$ and evaluation of mean-upper semi deviation risk measure $\rho(\cdot)$. The approximation in (\ref{final}) suggests that the VMS increases with the number of periods $T$ and the variability in the net load $\Delta$. However, VMS decreases with the degree of risk aversion $\lambda$.

\section{Computational Experiments} \label{sec:computation}
The analytical result of the previous section rely on restrictive assumptions to simplify  the structure of the RA-UC problem. In order to see how the VMS behave in the absence of these assumptions, we conduct a set of computational experiments next.

We consider a power system with 10 generators in the computational experiments. We use the data set presented in \cite{kazarlis1996genetic} with some modifications. We also consider a random net load process with eight scenarios where the power demand at each hour is subject to uncertainty. The scenario tree depicting the random process  is given \figurename{\ref{stree}}. A similar scenario tree structure is used in \cite{shiina2004stochastic}. 
\begin{figure}[!h]
	\centering
	\caption{Scenario tree \label{stree}}
	\includegraphics[width=0.35\textwidth]{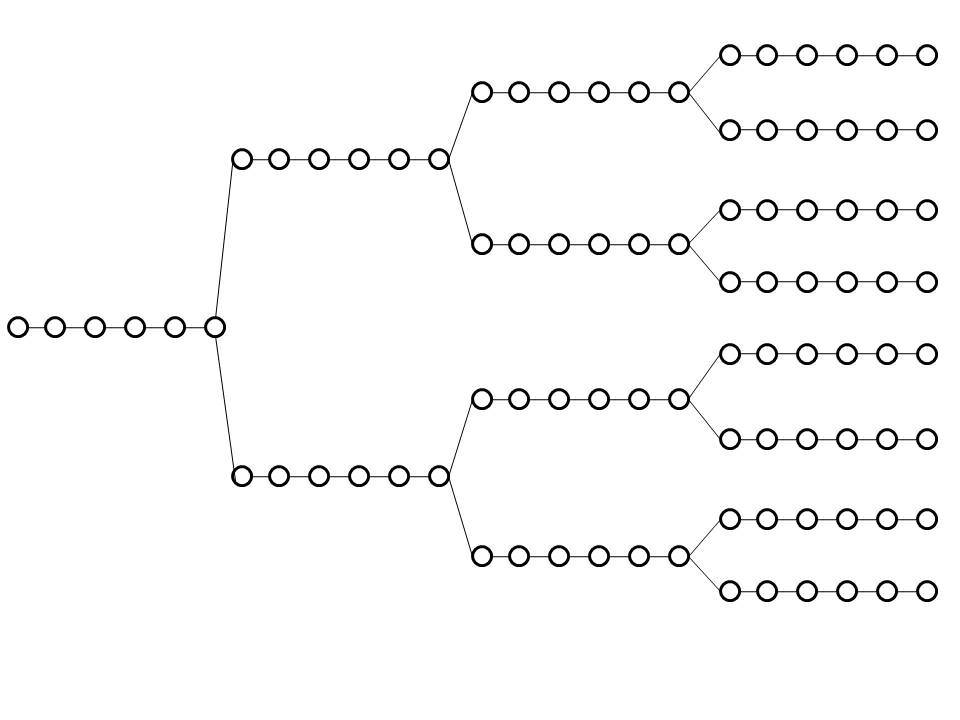}	
\end{figure}

The test data is presented in Appendix \ref{app:data}. We use the base demand values presented in Table \ref{tab:dem} to generate random demands. A variability parameter $\epsilon$ is used to control the dispersion of demand across all scenarios. Demand values for each scenario are presented in Table \ref{tab:sce}. All other parameters are set to the values presented in Table \ref{tab:gen}. A PC with two 2.2GHz processors and 6 GB of RAM is used in the computational experiments.

The quadratic production cost functions $\{g_i(\cdot)\}_{i \in \mathcal{I}}$ are approximated by a piecewise linear cost function with four pieces of equal lengths. We use a conditional mean-upper semi deviation risk measure (\ref{musd}) in each period.  The conditional risk measures $\rho_{1}(\cdot),\rho_2(\cdot),\ldots,\rho_{T-1}(\cdot)$, the dynamic coherent risk measure $\varrho(\cdot)$ and the composite risk measure $\rho(\cdot)$ are defined accordingly. 

We model and solve the two-stage model TS and the multi-stage model MS for five different values of variability parameter $\epsilon$ and six different values of the penalty parameter $\lambda$. For each $\epsilon$ and $\lambda$ pair, we calculate VMS in terms of difference of optimal values, that is,
\begin{equation*}
\text{VMS (\$)} = z^{TS} - z^{MS},
\end{equation*}
and in terms of percentage 
\begin{equation*}
\text{VMS (\%)} = \frac{z^{TS} - z^{MS}}{z^{MS}},
\end{equation*}
The results on the VMS are presented in \figurename{\ref{fig:VMS}}. 
\begin{figure*}
	\caption{\small Results of the computational experiments on the VMS. \label{fig:VMS}} 	
	\centering
	\begin{subfigure}[b]{0.485\textwidth}
		\centering
		\includegraphics[width=\textwidth]{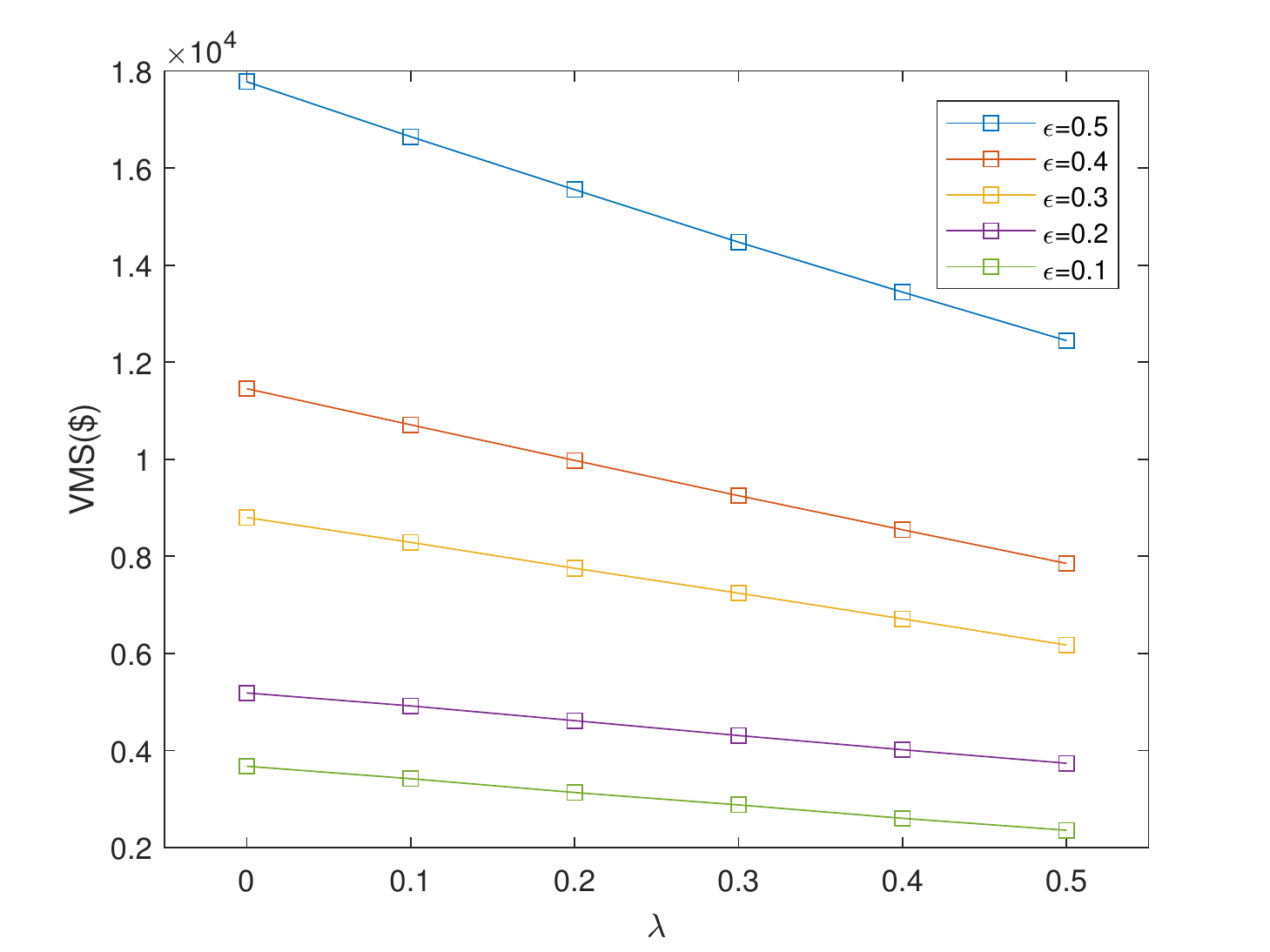}
	\end{subfigure}
	\hfill
	\begin{subfigure}[b]{0.485\textwidth}  
		\centering 
		\includegraphics[width=\textwidth]{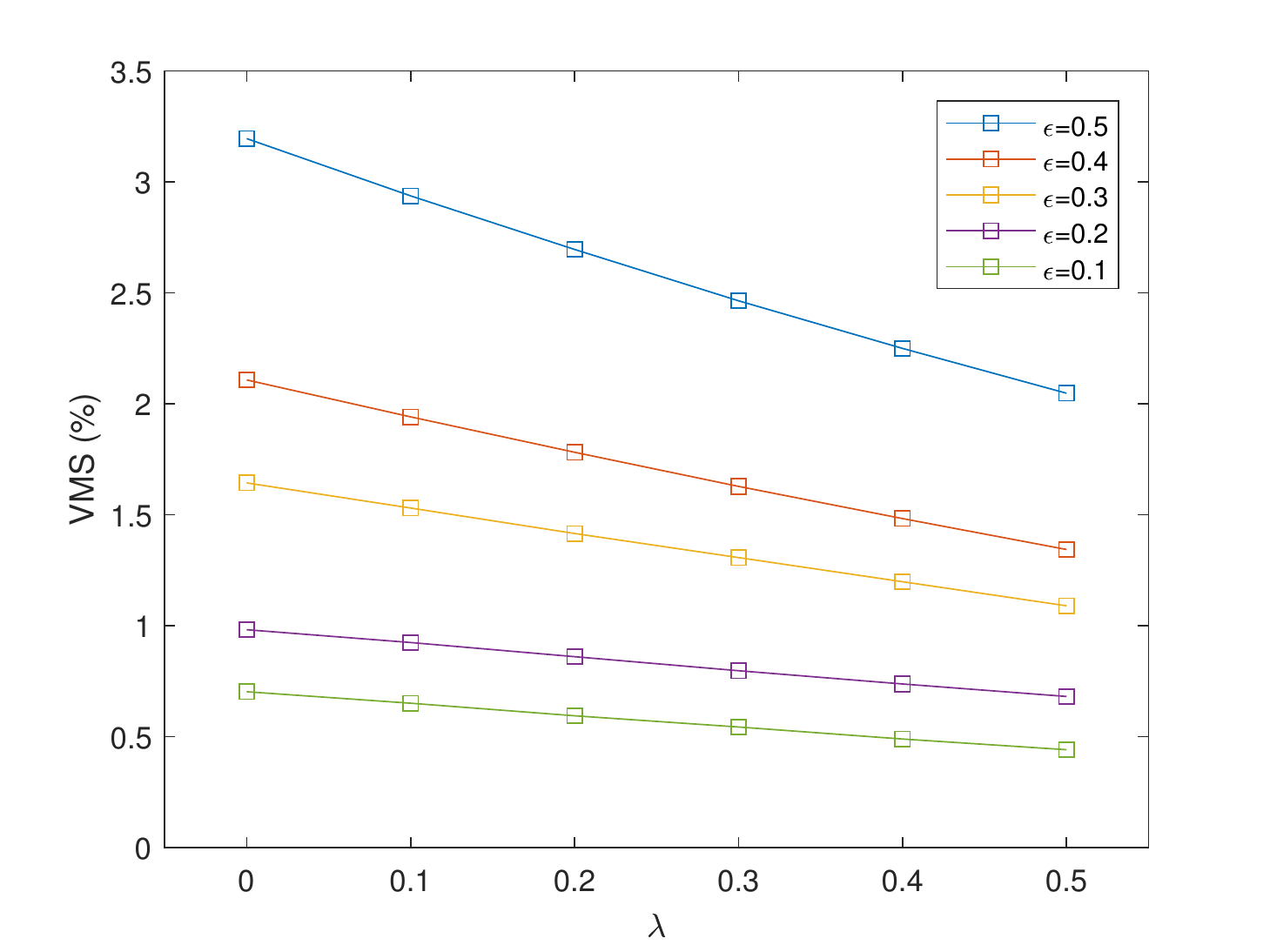}
	\end{subfigure}
\end{figure*}

 \figurename{\ref{fig:VMS}} verifies our analytical findings on  VMS.  We observe an increase in  VMS with the uncertainty in net load values. The VMS and hence importance of the multi-stage model increases as the dispersion among the scenarios increases. As expected, the day-ahead schedule obtained by solving the multi-stage model is more adaptive and provides more flexibility in case of high variability of problem parameters. 
We also observe decrease in the VMS with the level of risk aversion. In parallel with the analytical results in Theorem \ref{thepro}, higher risk aversion leads lower VMS. Hence, the importance of the multi-stage model decreases as risk aversion increases.

We also consider a rolling horizon policy obtained by solving two-stage approximations to the multi-stage problem in each period and fixing the decisions at that stage with respect to the optimal solution of the two-stage model.  In order to the measure the quality of the rolling horizon policy, we calculate the gap between the value of the rolling horizon policy and the optimal value of MS. The gap value GAP is calculated in terms of difference of objective values
\begin{equation*}
\text{GAP (\$)} = z^{RH} - z^{MS},
\end{equation*}
and in terms of percentage 
\begin{equation*}
\text{GAP (\%)} = \frac{z^{RH} - z^{MS}}{z^{MS}}.
\end{equation*}
where $z^{RH}$ is the value of the rolling horizon policy. Note that since rolling horizon provides a feasible policy to the multistage problem that is at least as good as that of TS, we have that $0 \leq \text{GAP}  \leq\text{VMS} $.  The results are presented in \figurename{\ref{fig:GAP}}. 

\begin{figure*}
	\centering
		\caption{\small Results of the computational experiments on GAP. \label{fig:GAP}} 
	\begin{subfigure}[b]{0.485\textwidth}
		\centering
		\includegraphics[width=\textwidth]{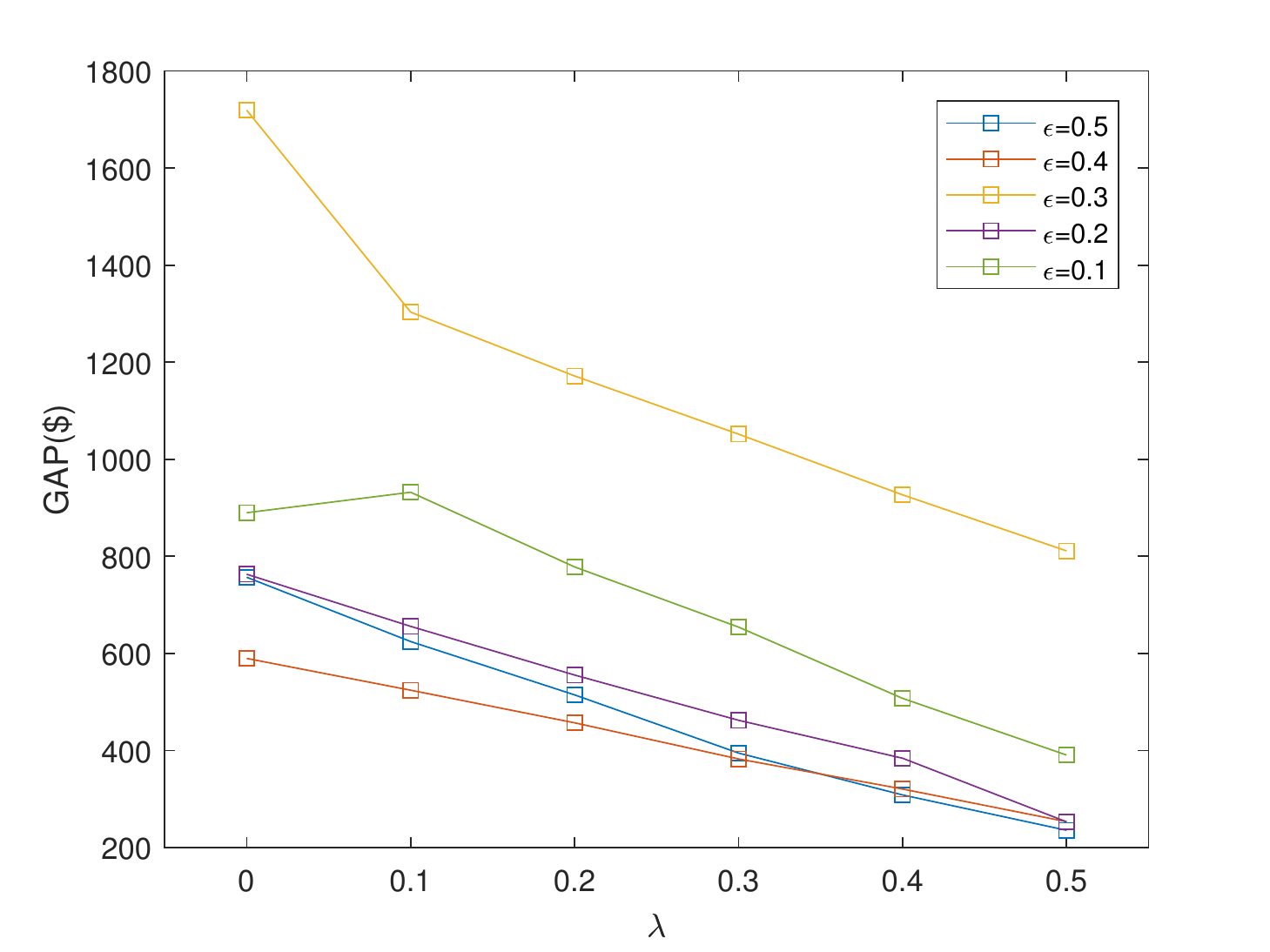}
	\end{subfigure}
	\hfill
	\begin{subfigure}[b]{0.485\textwidth}  
		\centering 
		\includegraphics[width=\textwidth]{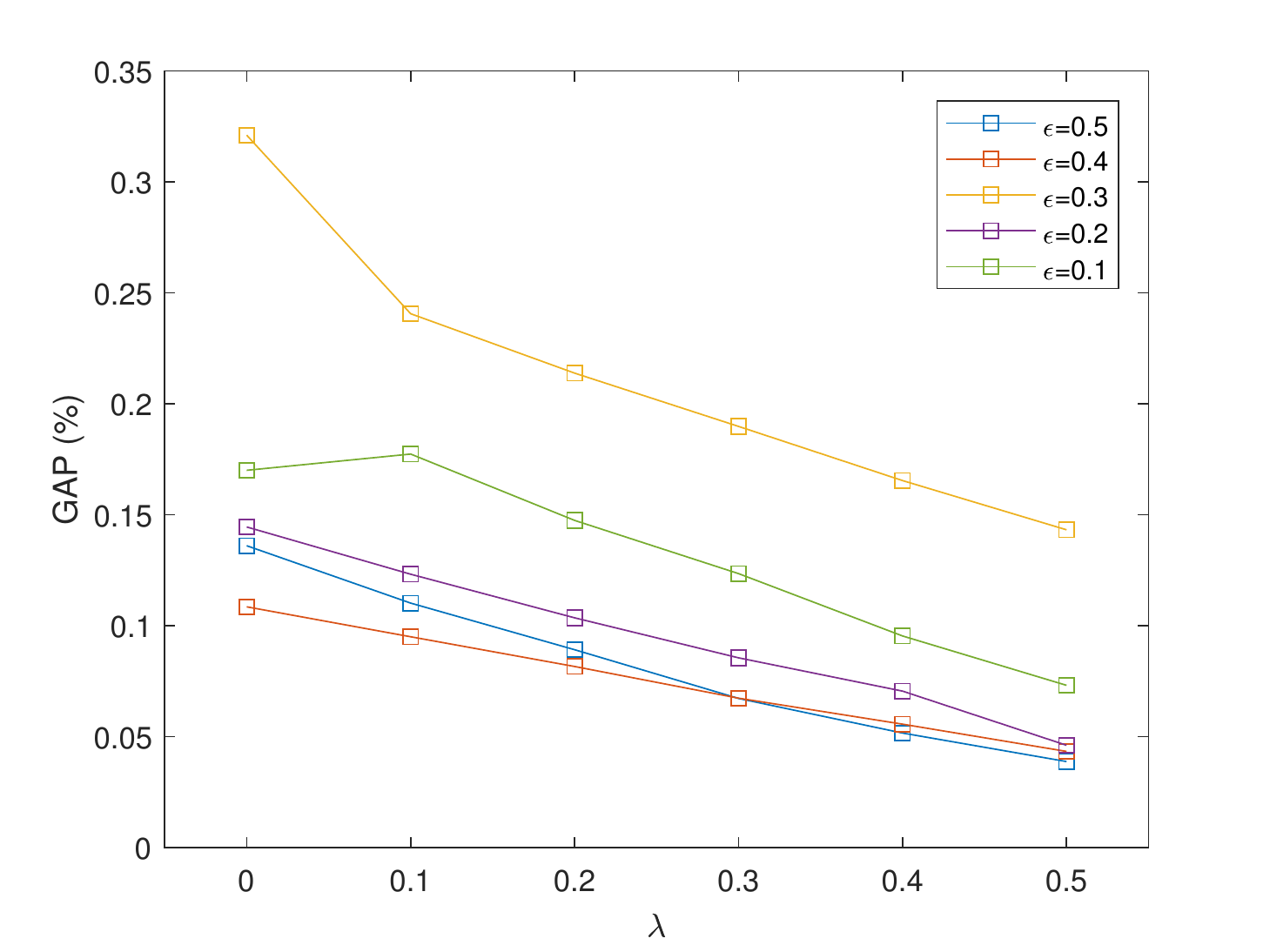}
	\end{subfigure}	
\end{figure*}

We present the solution times for each TS and MS instance at Table \ref{tab:timets} and Table \ref{tab:timems}, respectively. The required time to obtain the rolling horizon policy is also presented in Table \ref{tab:timerh}.

\begin{table}[!h]
	\centering 
	\caption{Solution times of TS (in seconds)}
	\begin{tabular}{c|cccccc} 
		$\epsilon \backslash  \lambda$   & 0     & 0.1   & 0.2   & 0.3   & 0.4   & 0.5 \\ \hline
		0.1   & 7.5  & 10.4 & 9.6  & 7.7  & 7.2  & 7.2 \\
		0.2   & 4.2  & 3.8  & 3.5  & 4.0  & 3.7  & 3.2 \\
		0.3   & 12.2 & 10.9 & 9.5  & 8.1  & 7.8  & 6.0 \\
		0.4   & 7.9  & 3.8  & 4.1  & 4.0  & 3.3  & 2.7 \\
		0.5   & 8.8  & 5.4  & 6.3  & 4.8  & 4.8  & 4.6 \\
	\end{tabular}%
	\label{tab:timets}%
\end{table}
\begin{table}[!h]
	\centering
	\caption{Solution times of MS (in seconds)}
	\begin{tabular}{c|cccccc} 
		$\epsilon \backslash  \lambda$   & 0     & 0.1   & 0.2   & 0.3   & 0.4   & 0.5 \\ \hline
		0.1   & 1004.2 & 1280.0 & 1255.2 & 1489.7 & 1789.6 & 2009.1 \\
		0.2   & 328.3  & 381.6  & 400.4  & 444.7 & 324.6 & 393.8 \\
		0.3   & 480.0  & 1042.4 & 435.8  & 780.0 & 453.8 & 358.5 \\
		0.4   & 192.9  & 674.5  & 529.4  & 323.0 & 328.6 & 279.8 \\
		0.5   & 85.7   & 147.5  & 116.6  & 119.0 & 118.5 & 113.1 \\
	\end{tabular}%
	\label{tab:timems}%
\end{table}
\begin{table}[!h]
	\centering
	\caption{Required time to obtain the rolling horizon policy (in seconds)}
	\begin{tabular}{c|cccccc} 
		$\epsilon \backslash  \lambda$   & 0     & 0.1   & 0.2   & 0.3   & 0.4   & 0.5 \\ \hline
	0.1	&	16.6	&	15.1	&	14.7	&	13.6	&	14.9	&	12.8	\\
	0.2	&	8.0	&	9.0	&	9.0	&	8.7	&	8.0	&	8.5	\\
	0.3	&	15.1	&	17.3	&	15.1	&	15.2	&	14.6	&	11.4	\\
	0.4	&	9.0	&	10.4	&	8.3	&	9.1	&	7.7	&	7.8	\\
	0.5	&	10.2	&	9.6	&	9.0	&	12.3	&	9.7	&	9.5	\\
	
	\end{tabular}%
	\label{tab:timerh}%
\end{table}

In all instances, the rolling horizon policy performs much better than the policy obtained by solving the two-stage problem with a small increase in computational effort. The $\text{GAP (\%)}$ of rolling horizon policy is $0.12\%$ on the average (with maximum $0.32\%$) whereas the $\text{VMS (\%)}$ is $1.42\%$ on the average (with maximum $3.20\%$). Thus, the rolling horizon policy obtained by using two-stage approximations to the multi-stage solution can provide enough flexibility in generation schedule to obtain a near-optimal schedule in RA-UC problems with a reasonable computational effort. 

The computational effort to solve the MS model is much larger than that of the TS model and the rolling horizon policy in all instances. The higher the demand variability leads higher VMS while decreasing the solution times as an additional benefit. 

\section{Conclusion} \label{sec:conclusion}
Recent improvements in the renewable power production technologies have motivated the stochastic unit commitment problems, since these models can explicitly address the variability in net load. Multi-stage models  provide completely flexible schedules where all decisions are adapted to the uncertainty. However, these models require high computational effort, and therefore, their two-stage counterparts are used to obtain approximate policies. In order to justify the additional effort to solve the multi-stage model rather than its two-stage counterpart, we define the VMS and provide analytical and computational results on it. These results reveal that, for RA-UC problems, the VMS decreases with the degree of risk aversion, and increases with the level of uncertainty and number of time periods.  

Performance of the rolling horizon polices obtained by two-stage approximations of the multi-stage models are promising. As a future research direction, it would be interesting to consider the rolling horizon policies in instances with more complicated random net load processes. However, in that case, the number of two-stage models to be solved would be large and their solution would require significant computation time.  Theoretical analysis of the value of rolling horizon policies is also an important future step.

\appendices
\section{Deterministic Unit Commitment Formulation} \label{app:model}
\emph{Indexes and Sets} 
\begin{align}
t: \; & \text{Period  index, } & i: \; & \text{Generator index}, \nonumber\\
T: \; & \text{Number of periods, } & I: \; & \text{Number of generators}, \nonumber\\
\mathcal{T}: \; & \text{Set of periods, } & \mathcal{I}: \; & \text{Set of generators}, \nonumber
\end{align}
\emph{Parameters} 
\begin{align}
a_{i}: \; & \text{Fixed cost of running generator } i \in \mathcal{I}, \nonumber \\
g_{i}(\cdot): \; & \text{Production cost function of running generator } i \in \mathcal{I},  \nonumber \\
&  \text{ specifically, } g_{i}(v) = b_{i}v+c_{i}v^2  \text{ for } v \geq 0 \nonumber \\
&  \text{ with parameters } b_i,c_i \in \mathbb{R}_+, \nonumber \\
SU_{i}: \; & \text{Start-up cost of generator } i \in \mathcal{I}, \nonumber \\
SD_{i}: \; & \text{Shut-down cost of generator } i \in \mathcal{I}, \nonumber \\
\underline{q}_{i}: \; & \text{Minimum production amount of generator } i \in  \mathcal{I}, \nonumber \\
\overline{q}_{i}: \; & \text{Maximum production amount of generator } i \in  \mathcal{I}, \nonumber \\
d_{t}: \; & \text{Net load in period } t \in \mathcal{T}, \nonumber \\
M_i: \; & \text{Minimum up time of generator } i \in  \mathcal{I}, \nonumber \\
L_i: \; & \text{Minimum down time of generator } i \in  \mathcal{I}, \nonumber \\
V'_i: \; & \text{Start up rate of generator } i \in \mathcal{I}, \nonumber \\
V_i: \; & \text{Ramp up rate of generator } i \in  \mathcal{I}, \nonumber \\
B'_i: \; & \text{Shut down rate of generator } i \in \mathcal{I}, \nonumber \\
B_i: \; & \text{Ramp down production limit of generator } i \in  \mathcal{I}.  \nonumber 
\end{align}
\emph{Variables}
\begin{align}
u_{it}: \; &  \text{Status of generator } i \in \mathcal{I} \text{ in period } t \in \mathcal{T},  \nonumber \\ 
\; & (1  \text{ if generator $i$ is ON in period $t$; } 0 \text{ otherwise}),  \nonumber \\
v_{it}: \; &  \text{Production amount of generator } i \in \mathcal{I} \text{ in period } t \in \mathcal{T}, \nonumber \\
y_{it}: \; &  \text{Start up decision of generator } i \in  \mathcal{I}  \text{ in period } t \in \mathcal{T},  \nonumber \\ 
&    (1  \text{ if } u_{i(t-1)} = 0  \text{ and } u_{it} = 1 \text{; } 0 \text{ otherwise}), \nonumber\\
z_{it}: \; &  \text{Shut down decision  of generator } i \in  \mathcal{I}  \text{ in period } t \in \mathcal{T},  \nonumber \\ 
&   (1  \text{ if } u_{i(t-1)} = 1  \text{ and } u_{it} = 0 \text{; } 0 \text{ otherwise}). \nonumber
\end{align}
\emph{Model}\\
\begin{align}
\underset{u, v, y ,z}{\text{min}} \; &  \sum_{t = 1}^T \sum_{i = 1}^I a_{i}u_{it} + g_{t}(v_{it}) + SU_{i}y_{it} + SD_{i}z_{it}, \label{deter:obj} \\
\text{s.t.} \; &  (\ref{det-dem}), (\ref{det-cap})  \nonumber\\
& u_{it}-u_{i(t-1)} \leq u_{i \tau}, \;   \forall t \in \mathcal{T},  \forall i \in \mathcal{I}, \nonumber   \\
& \quad \forall \tau \in \{t+1,\ldots,\min\{t+M_i,T\}\} \label{deter:upt}   \\
& u_{i(t-1)} - u_{it} \leq 1 - u_{i\tau}, \;  \forall t \in \mathcal{T}, \forall i \in \mathcal{I}, \nonumber \\
& \quad \forall \tau \in \{t+1,\ldots,\min\{t+L_i,T\}\} \label{deter:dot}   \\
& u_{it}-u_{i(t-1)} \leq y_{it}, \;  \forall t \in \mathcal{T}, \forall i \in \mathcal{I} \label{deter:sup}  \\
& u_{i(t-1)}-u_{it} \leq z_{it} , \;  \forall t \in \mathcal{T}, \forall i \in \mathcal{I} \label{deter:sdo}  \\
& v_{it}-v_{i(t-1)} \leq V'_i y_{it} + V_i u_{i(t-1)},  \nonumber \\
& \quad  \forall t \in \mathcal{T}, \forall i \in \mathcal{I} \label{deter:rup}  \\
& v_{i(t-1)}-v_{it} \leq B'_i z_{it} + B_i u_{it}, \nonumber \\
& \quad \forall t \in \mathcal{T}, \forall i \in \mathcal{I} \label{deter:rdo}  \\
& u_{it}, y_{it}, z_{it} \in \{0,1\}, v_{ti} \geq 0, \;  \forall t \in \mathcal{T}, \forall i \in \mathcal{I}. \nonumber 
\end{align}
The objective (\ref{deter:obj}) is total fixed, production, start up and shut down costs. Constraints (\ref{deter:upt}), (\ref{deter:dot}), (\ref{deter:sup}) and (\ref{deter:sdo}) are minimum up time, minimum down time, start up and  shut down constraints, respectively. The rump/start up rate constraint is given in (\ref{deter:rup}). Similarly, (\ref{deter:rdo}) is the rump/shut down rate constraint. 

\section{Computational Experiment Data} \label{app:data}
\begin{table}[!h]
	\centering
	\caption{Demand Data (MW = megawatt)}
	\begin{tabular}{|c|c|c|c|c|c|c|} \hline
		$t$  & 1     & 2     & 3     & 4     & 5     & 6     \\ \hline
		$\overline{d}_t$ (MW) & 700   & 750   & 850   & 950   & 1000  & 1100  \\ \hline
		$t$  & 7     & 8     & 9     & 10    & 11    & 12 \\ \hline
		$\overline{d}_t$ (MW)  & 1150  & 1200  & 1300  & 1400  & 1450  & 1500 \\ \hline
		$t$  & 13    & 14    & 15    & 16    & 17    & 18    \\ \hline
		$\overline{d}_t$  (MW) & 1400  & 1300  & 1200  & 1050  & 1000  & 1100 \\ \hline
		$t$  & 19    & 20    & 21    & 22    & 23    & 24 \\ \hline
		$\overline{d}_t$  (MW) & 1200  & 1400  & 1300  & 1100  & 900   & 800 \\ \hline
	\end{tabular}%
	\label{tab:dem}%
\end{table}
\begin{table}[!h]
	\centering
	\caption{Scenario Data}
	\begin{tabular}{cccccc}
		&       & \multicolumn{4}{c}{Period (or hour) $t$} \\ \cline{3-6}
		\multicolumn{1}{l}{Scenario} & \multicolumn{1}{l}{Probability} &   1-6    & 7-12    & 13-18      & 19-24  \\ \hline
		1     & 0.125 & $\overline{d}_t$     & $(1-\epsilon)\overline{d}_t$     & $(1-\epsilon)\overline{d}_t$     & $(1-\epsilon)\overline{d}_t$  \\
		2     & 0.125 & $\overline{d}_t$     & $(1-\epsilon)\overline{d}_t$     & $(1-\epsilon)\overline{d}_t$     & $(1+\epsilon)\overline{d}_t$ \\
		3     & 0.125 & $\overline{d}_t$     & $(1-\epsilon)\overline{d}_t$     & $(1+\epsilon)\overline{d}_t$     & $(1-\epsilon)\overline{d}_t$  \\
		4     & 0.125 & $\overline{d}_t$     & $(1-\epsilon)\overline{d}_t$     & $(1+\epsilon)\overline{d}_t$     & $(1+\epsilon)\overline{d}_t$ \\
		5     & 0.125 & $\overline{d}_t$     & $(1+\epsilon)\overline{d}_t$     & $(1-\epsilon)\overline{d}_t$     & $(1-\epsilon)\overline{d}_t$  \\
		6     & 0.125 & $\overline{d}_t$     & $(1+\epsilon)\overline{d}_t$     & $(1-\epsilon)\overline{d}_t$     & $(1+\epsilon)\overline{d}_t$ \\
		7     & 0.125 & $\overline{d}_t$     & $(1+\epsilon)\overline{d}_t$     & $(1+\epsilon)\overline{d}_t$     & $(1-\epsilon)\overline{d}_t$  \\
		8     & 0.125 & $\overline{d}_t$     & $(1+\epsilon)\overline{d}_t$     & $(1+\epsilon)\overline{d}_t$     & $(1+\epsilon)\overline{d}_t$ \\
	\end{tabular}%
	\label{tab:sce}%
\end{table}%
\begin{table}[!h]
	\centering
	\caption{Generator Data (MW = megawatt)}
	\begin{tabular}{|c|c|c|c|c|c|} \hline
		$i$  & 1     & 2     & 3     & 4     & 5      \\ \hline 
		$a_i$ (\$/h)  & 1000  & 970   & 700   & 680   & 450   \\ \hline
		$b_i$ (\$/MWh)  & 16.19 & 17.26 & 16.6  & 16.5  & 19.7  \\ \hline
		$c_i$ (\$/MW$^{2}$h)   & 0.00048 & 0.00031 & 0.002 & 0.00211 & 0.00398  \\ \hline
		$\overline{q}_i$ (MW) & 682.5	&	682.5	&	195	&	195	&	243			\\ \hline
		$\underline{q}_i$ (MW)  & 225	&	225	&	30	&	30	&	37.5			\\ \hline
		$V'_i$ (MW)  & 337.5	&	337.5	&	45	&	45	&	56.25			\\ \hline
		$V_i$ (MW)  & 405	&	405	&	54	&	54	&	67.5	\\ \hline
		$B'_i$ (MW)  & 337.5	&	337.5	&	45	&	45	&	56.25		\\ \hline
		$B_i$ (MW)  & 405	&	405	&	54	&	54	&	67.5	
		\\ \hline
		$M_i$ (h) & 8     & 8     & 5     & 5     & 6     \\ \hline
		$L_i$ (h) & 8     & 8     & 5     & 5     & 6     \\ \hline
		$SU_{i}$ (\$/h) & 4500  & 5000  & 550   & 560   & 900   \\ \hline 
		$SD_{i}$ (\$/h) & 0  & 0  & 0   & 0   &  0   \\ \hline\hline
		$i$    & 6     & 7     & 8     & 9     & 10 \\ \hline 
		$a_i$ (\$/h)   & 370   & 480   & 660   & 665   & 670 \\ \hline
		$b_i$ (\$/MWh)   & 22.26 & 27.74 & 25.92 & 27.27 & 27.79 \\ \hline
		$c_i$ (\$/MW$^{2}$h)   & 0.00712 & 0.00079 & 0.00413 & 0.00222 & 0.00173 \\ \hline
		$\overline{q}_i$ (MW) 	&	120	&	127.5	&	82.5	&	82.5	&	82.5
		\\ \hline
		$\underline{q}_i$ (MW) 	&	30	&	37.5	&	15	&	15	&	15
		\\ \hline
		$V'_i$ (MW)  	&	45	&	56.25	&	22.5	&	22.5	&	22.5
		\\ \hline
		$V_i$ (MW) 	&	54	&	67.5	&	27	&	27	&	27
		\\ \hline
		$B'_i$ (MW) 	&	45	&	56.25	&	22.5	&	22.5	&	22.5
		\\ \hline
		$B_i$ (MW)  	&	54	&	67.5	&	27	&	27	&	27\\ \hline
		$M_i$ (h)     & 3     & 3     & 1     & 1     & 1 \\ \hline
		$L_i$ (h)     & 3     & 3     & 1     & 1     & 1 \\ \hline
		$SU_i$ (\$/h)   & 170   & 260   & 30    & 30    & 30 \\ \hline 
		$SD_i$ (\$/h)   & 0   & 0   & 0    & 0    & 0 \\ \hline
	\end{tabular}%
	\label{tab:gen}%
\end{table}%

\begin{IEEEbiographynophoto}{Ali~\.{I}rfan~Mahmuto\u{g}ullar{\i}}
	is a Ph.D. candidate in Department of Industrial Engineering, Bilkent University, Ankara, Turkey. He holds B.S. and M.S. degrees from the same department in 2011 and 2013, respectively. His research focuses on developing efficient solution methods for multi-stage mixed-integer stochastic programming models. He is also interested in application of these models to the problems emerging from different areas of operations research. 
\end{IEEEbiographynophoto}

\begin{IEEEbiographynophoto}{Shabbir~Ahmed} is the Anderson-Interface Chair and Professor in the H. Milton Stewart School of Industrial \& Systems Engineering at the Georgia Institute of Technology. His research interests are in stochastic and discrete optimization. Dr. Ahmed is a past Chair of the Stochastic Programming Society and serves on the editorial board of several journals. His honors include the INFORMS Computing Society Prize, the National Science Foundation CAREER award, two IBM Faculty Awards, and the INFORMS Dantzig Dissertation award. He is a Fellow of INFORMS.\end{IEEEbiographynophoto}

\begin{IEEEbiographynophoto}{{\"O}zlem~{\c{C}}avu{\c{s}}}
	is currently an Assistant Professor of Industrial Engineering at Bilkent University. She received her B.S. and M.S. degrees in Industrial Engineering from Boğaziçi University in 2004 and 2007, respectively, and the Ph.D. degree in Operations Research from Rutgers Center for Operations Research (RUTCOR) at Rutgers University in 2012. Her research interests include stochastic optimization, risk-averse optimization and Markov decision processes.
\end{IEEEbiographynophoto}

\begin{IEEEbiographynophoto}{M.~Selim~Akt{\"u}rk}
	is a Professor and Chair of the Department of Industrial Engineering at Bilkent University. His recent research interests are discrete optimization, production scheduling and airline disruption management.
\end{IEEEbiographynophoto}





\vfill

\end{document}